\documentclass[12 pt, reqno]{amsart}

\usepackage{amsthm,amssymb,amstext,amscd,amsfonts,amsbsy,amsxtra,latexsym,cite}
\usepackage{mathrsfs}
\usepackage{verbatim}
\usepackage{url}
\usepackage[latin1]{inputenc}
\usepackage{graphicx, graphics, wrapfig}
\usepackage[final]{changes}

\theoremstyle{plain}
\newtheorem{theorem}{Theorem}
\newtheorem*{theorem*}{Theorem}
\newtheorem{lemma}[theorem]{Lemma}
\newtheorem{proposition}[theorem]{Proposition}

\newtheorem*{conjecture*}{Conjecture}

\theoremstyle{definition}
\newtheorem*{definition}{Definition}

\theoremstyle{remark}
\newtheorem*{example}{Example}

\newtheorem*{remark*}{Remark}

\numberwithin{theorem}{section} \numberwithin{equation}{section}

\newcommand{\Z}{\mathbb{Z}}

\newcommand{\spa}{\hspace{5 mm}}
\newcommand{\cc}{\mathscr{C}}
\newcommand{\dd}{\mathscr{D}}

\title{Asymptotic Formulae for Mixed Congruence Stacks} 

\begin{document}
\author{Richard Frnka}
\address{Department of Mathematics \\
Louisiana State University \\
Baton Rouge, LA 70802\\ U.S.A.}
\email{richardfrnka AT outlook.com}

\date{\today}


\maketitle
\begin{abstract}
Much like the important work of Hardy and Ramanujan \cite{HR} proving the asymptotic formula for the partition function, Auluck \cite{Auluck} and Wright \cite{Wright2} gave similar formulas for unimodal sequences. Following the Circle Method of Wright, we provide the asymptotic expansion for unimodal sequences on a two-parameter family of mixed congruence relations, with parts on one side up to the peak satisfying $r \pmod{m}$ and parts on the other side $-r\pmod{m}$. Techniques used in the proofs include Wright's Circle Method, modular transformations, and bounding of complex integrals.
\end{abstract}

\section{Introduction and Statement of Main Results}
Hardy and Ramanujan \cite{HR} provided an asymptotic expansion of the partition function for large $n$, showing the exponential growth of $p(n)$:
 \[
p(n) \sim \frac{1}{2^2 3^{1/2} n} e^{\pi \sqrt{\frac{2n}{3}}}.
\]
When adding restrictions to the parts so that each part is $r \pmod{m}$ and letting $p_{(r,m)}(n)$ count these partitions, an asymptotic formula given by Meinardus \cite{Mein} is 
\[
p_{(r,m)}(n)\sim \left(\Gamma\left(\frac{r}{m}\right) \pi ^{(r/m)-1}2^{-(3/2)-(r/2m)}3^{-(r/2m)}m^{-1/2 + (r/2m)}\right) n^{-\frac12 \left(1+\frac{r}{m}\right)} e^{ \pi \sqrt{\frac{2n}{3m}}}.
\]
With the parts allowed to be $\pm r \pmod{m}$ and letting $p_{(\pm r,m)}(n)$ be the number of partitions, Livingood \cite{Living} showed
\[
p_{(\pm r,m)}(n)  \sim \frac{\csc{\frac{\pi r}{m}}}{4 \pi 3^{\frac14 }m^{\frac14}n^{\frac34}}e^{2\pi \sqrt{\frac{n}{3m}}}.
\]
The generating functions for these partition relations involve infinite $q$-series and manipulations of the Circle Method are used in the proofs. Beckwith and Mertens \cite{BM} have also proven other variations involving the multiplicity of parts appearing in residue classes.
\\ \indent
We extend these ideas to unimodal sequences, where the restriction on non-increasing ordering of parts for partitions is slightly relaxed. \begin{definition} A \emph{unimodal sequence} or \emph{stack} of size $n$ is a sequence of non-zero parts $a_1 \dots a_r$, $c$, and $b_s\dots b_1$ such that
\begin{equation}
\label{s1}
n= \sum_{i=1}^r a_i + c + \sum_{j =1}^s b_j 
\end{equation}
and
\begin{equation}
\label{s2}
1 \leq a_1 \leq a_2 \leq \dots \leq a_r \leq c > b_s \geq \dots \geq b_2 \geq b_1
\end{equation}
\end{definition}
\noindent We let $s(n)$ be the number of stacks of size $n$ and show the $8$ unimodal sequences of size $4$ in Figure \ref{stackPic}.
\begin{figure}[h]
\begin{center}
\includegraphics[scale=.5]{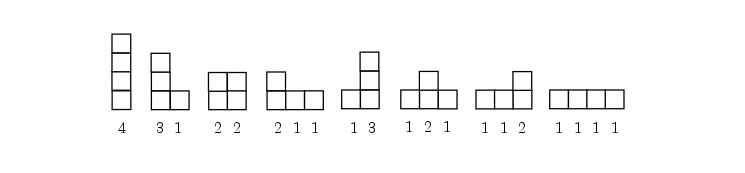} 
\caption{The stacks of size $4$}
\label{stackPic}
\end{center}
\end{figure}
\\
\indent Questions about the asymptotics for the stack function were first introduced by the physicist Temperley \cite{Temp} in 1951, who was studying particle configurations on specific lattices. The entropy of the system depends on the exponential bound of $s(n)$, which had not been previously discovered. This was provided by Auluck in 1951 \cite{Auluck}, when he showed
\begin{equation}
s(n) \sim \frac{1}{2^33^{3/4}n^{5/4}} e^{2\pi\sqrt{\frac{n}{3}}}. \label{stacky}
\end{equation}
This was proved using a different method by Wright in 1971 \cite{Wright2} and he also gave a way to calculate lower order terms, providing the asymptotic expansion in powers of $n^{-\frac12}$. 
\\
\indent
In \cite{Andrews3}, Andrews examined some interesting cases of concave and convex unimodal sequences. He defined a ``mixed parity" case where on one side of the stack the parts are odd, with even parts on the other side. A detailed study of the combinatorial and asymptotic behavior of unimodal sequences with parity conditions will be the subject of a subsequent paper.
\begin{definition}
A \emph{mixed congruence stack} is a unimodal sequence satisfying \eqref{s1} and \eqref{s2}, as well as congruence relations on the parts where for $m>1$, gcd$(r,m)=1$, and $0<r< m/2$,
\[
a_i \equiv r (\text{mod } m),\spa c \equiv r (\text{mod } m), \spa b_j \equiv -r (\text{mod } m)\spa \hbox{for all }i \text{ and } j,
\]  
\end{definition}
Letting $s_{(r,m)}(n)$ be the number of mixed congruence stacks of size $n$, the generating function is
\begin{equation}
\label{MC}
S_{(r,m)}(q) := \sum_{ n \geq 1 } s_{(r,m)}(n) q^n = \sum_{ k \geq 0} \frac{ q^{km+r}}{(q^r;q^m)_{k+1}(q^{m-r};q^m)}_k,
\end{equation}
where we are following the standard $q$-series notation 
\[
(a;q)_n = \prod_{ i =0}^{n-1} (1-aq^i).
\]
\indent The $k$ index in the generating function refers to the height of the peak at $km+r$. Then we have parts up to $km+r$ on the left and parts up to $km-r$ on the right of the stack.  Shown in Figure \ref{MixedS}, we see the mixed stacks for the relation $1 \pmod{4}$ of size $12$. 
\begin{figure}[h]
\centering
\includegraphics[scale=.5]{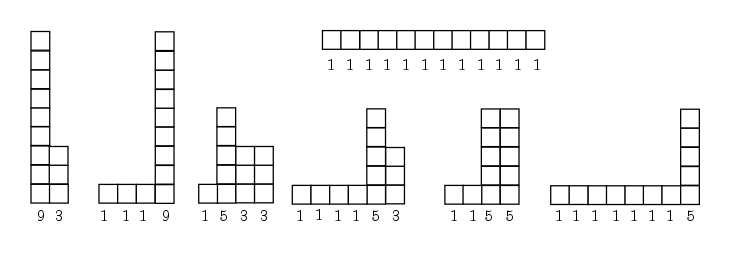}
\caption{Mixed congruence stacks for $s_{(1,4)}(12)$}
\label{MixedS}
\end{figure}

\begin{theorem}
For $0<r<m/2$,  
\label{Theorem1}
\[
s_{(r,m)}(n)  \sim \frac{\csc{\left( \frac{\pi r}{m}\right)}}{2^3 3^{1/4} m^{1/4}n^{3/4}} e^{2\pi\sqrt{\frac{n}{3m}}}
\]
\end{theorem}
\begin{remark*} Note that Theorem \ref{Theorem1} applies for $0< r< m$, but when $r>m/2$, there is a missing part of the congruence on the right side of the stack. In this case, the term $mk+(m-r)$ is less than the peak at $mk+r$ but is absent, creating a gap. The generating function that fills the gap when $r>m/2$ can be analyzed separately
\begin{equation}
S_{(r,m)}^*(q) := \sum_{ k \geq 0} \frac{ q^{km+r}}{(q^r;q^m)_{k+1}(q^{m-r};q^m)}_{k+1}.
\end{equation}
\end{remark*}
\begin{remark*}
\indent In the case $m=1$, a pole is introduced so it is excluded by adding the requirement $m>1$. Note that the exponential term of Theorem \ref{Theorem1} for $m=1$ matches \eqref{stacky} although the constant terms are different. 
\end{remark*} 
\begin{remark*}
The proof of Theorem \ref{Theorem1} uses the Circle Method developed by Wright in \cite{Wright2} and \cite{WrightAS}. It is important to note that Ngo and Rhoades have recently proven a general version of Wright's approach. Propositions 1.8 - 1.10 of \cite{Road} are widely applicable to  functions of the shape $\xi(q) L(q)$, where $\xi$ has an exponential singularity near $q=1$, and $L(q)$ has an asymptotic expansion. Read in the context of Proposition 1.8 from \cite{Road}, the bulk of the work in Section \ref{lastSection} of this paper is devoted to verifying HYPO 4 from Ngo-Rhoades, as the infinite product that occurs here must be analyzed separately near roots of unity.
\end{remark*}
\begin{example} For the case $m=3$ and $r=1$, we have 
\[
S_{(1,3)}(q)= \sum_{n \geq 1} s_{(1,3)}(n)q^n = \sum_{ k \geq 0} \frac{q^{3k+1}}{(q;q^3)_{k+1} (q^2;q^3)_{k}}.
\]
Letting the first term in our asymptotic formula be
\[
X_{(1,3)}(n) = \frac{1}{2^2 \cdot 3 \cdot n^{3/4}}e^{\left(\frac{2\pi }{3}\sqrt{n}\right)},
\] 
we can see the accuracy of the theorem statement in Table \ref{freshTab} based on relative error.
\begin{table}[h]
\caption{Large values of $s_{(1,3)}(n)$ versus the asymptotic formula}
\label{freshTab}
\[
\begin{array}{|c|c|c|c|}
\hline
 n& s_{(1,3)}(n) & X_{(1,3)}(n) & (X_{(1,3)}(n)-s_{(1,3)})/(s_{(1,3)}(n)) \\
\hline 100 & {3.1671 \times 10^{6}} & {3.2859 \times 10^{6}} & .03751 \\
\hline 1000 &{ 2.6926 \times 10^{25}} & 	{2.7189 \times 10^{25}} & .00977 \\ 
\hline 10000 &{ 7.5714 \times 10^{86}} & {7.5726 \times 10^{86}} &.00016 \\
\hline
\end{array}
\]
\end{table}
\end{example}

\indent The generating function is rewritten in terms of an infinite $q$-series product and a false theta function. A modular transformation is established for the $q$-series and an asymptotic expansion is given for the false theta function. Details are in Section \ref{nextSection}. 
\\
\indent Once we have the transformation formula, Section \ref{lastSection} proceeds to evaluate the coefficients  by Cauchy's formula. Wright's Circle Method is used to give the asymptotic formula for the integral and thus mixed congruence stacks.
\section*{Acknowledgements} This work is a result of the author's Ph.D. Thesis at Louisiana State University. Special thanks to Karl Mahlburg for advising on the project, providing knowledge on the subject, and pointing to excellent references. The author also thanks the two anonymous referees for their many helpful comments and thoroughness in reviewing this paper.
\section{Hypergeometric $q$-series and Theta Functions}\label{nextSection}
As in the work of Wright \cite{Wright2} on ordinary unimodal sequences, we must transform the generating function.
\begin{proposition} 
\begin{align}
S_{(r,m)}(q) &= \frac{1}{(q^r,q^{m-r};q^m)_{\infty}}\sum_{n \geq 0} (-1)^n q^{\frac{mn(n+1)}{2}-2rn} \nonumber \\
 &+\sum_{ n \geq 0} (-1)^{n-1} q^{\frac{mn(3n+1)}{2} -3rn}\left(1-q^{(2n+1)m-2r}\right) .\label{MixedGF}
\end{align}
\end{proposition}
\begin{proof}
One of the equations from Ramanunjan's Lost Notebook representing the unimodal-rank generating function is explored in depth in \cite{Andrews} as Equation (1.1): 
\begin{align}
1 + \sum_{ n \geq 1} \frac{q^n}{(aq;q)_n(a^{-1}q;q)_n} &= (1+a) \sum_{ n \geq 0} a^{3n} q^{ \frac{n}{2} (3n+1)} \left( 1-a^2 q^{2n+1}\right)\nonumber \\
&\spa \spa \spa \spa - a \sum_{ n \geq 0}\frac{ (-1)^n a^{2n} q^{\frac{n}{2}(n+1)}}{(aq;q)_{\infty}(a^{-1}q;q)_{\infty}}.\label{supere}
\end{align}
This equation is a special case of Heine's transformation formula for hypergeometric series.
Sending $q \mapsto q^m$ and $a \mapsto -q^{-r}$, multiplying by $\frac{q^r}{(1-q^r)}$ and reorganizing in \eqref{supere} completes the proof.
\end{proof}
We wish to analyze individually the different pieces of the right hand side of \eqref{MixedGF}, so we let 
\[
S_{(r,m)}(q ) = F(q)L(q) + R(q)
\]
where 
\begin{align*}
F(q) := &(q^r,q^{m-r};q^m)_{\infty}^{-1}, \spa L(q) := \sum_{n \geq 0} (-1)^n q^{\frac{mn(n+1)}{2}-2rn},\\
\text{and }\spa&R(q) := \sum_{ n \geq 0} (-1)^{n-1} q^{\frac{mn(3n+1)}{2} -3rn}\left(1-q^{(2n+1)m-2r}\right)
\end{align*}
The $R(q)$ term only offers a contribution of $\pm 1$ for high order terms, so letting $H(q):= F(q)L(q)$ be the main term, we have that $s_{(r,m)}(n) \sim h(n)$ for $H(q) = \sum_{ n \geq 1}h(n)q^n$.
\\
\indent With the generating series in this particular form, we can turn our attention to each individual piece. The sum $L(q)$ is a false theta function whose coefficients can be bounded, and the infinite $q$-series satisfies a modular transformation formula. The transformation formula will provide a first term in the asymptotic expansion of $F(q)$ that will be vital for the circle and function over which we will be integrating. 
\begin{proposition} For $q=e^{2\pi i \tau}$ and $0<r <m/2$, as $\tau \to 0$,
\begin{align}
\frac{1}{\left( q^r,q^{m-r}; q^m\right)_{\infty}} = \frac{1}{2\sin( \frac{\pi(m-r)}{m})} e^{\left( \frac{-\pi i r(m-r)}{m} + \frac{ \pi i m}{6}\right)\tau + \frac{ \pi i}{6m \tau}}
 \left(1+O\left( e^{\frac{-2\pi i}{m \tau}}\right)\right).
\end{align}
\label{PROP1}

\end{proposition}
\begin{proof}

The function $F(q)$ can be expressed in terms of Jacobi's theta function (see \cite{Dave} for a good reference), which is 
\begin{equation}
 \Theta(w; \tau) := \sum_{n \in \frac12 +\Z} e^{\pi i n^2 \tau + 2\pi i n ( w + \frac12)},
\end{equation}
and has an equivalent product form by the Triple Product Identity, 
\begin{equation}
\Theta(w;\tau) = -i q^{\frac18} e^{-\pi i w}(q;q)_{\infty} (e^{2\pi i w}; q)_{\infty} (e^{-2\pi i w}q;q)_{\infty}.
\label{jacobi}
\end{equation}
The theta function satisfies the following  transformation formulas:
\begin{align}
\Theta(-w;\tau) &= - \Theta(w;\tau),\\
\Theta\left(\frac{w}{\tau}; -\frac{1}{\tau} \right) &= -i \sqrt{-i\tau} e^{\frac{ \pi i w^2}{\tau}}\Theta(w;\tau).
\label{thetatrans}
\end{align}
\noindent The Dedekind eta function satisfies the inversion formula 
\begin{equation}
\label{etatrans}
\eta\left(-\frac{1}{\tau}\right)  = \sqrt{-i \tau} \eta(\tau). 
\end{equation}
We can express our function in terms of the theta functions under a nice transformation.  \\
\indent 
We have for  $m\geq 2$ and $ 1\leq r< \frac{m}{2} $, that 
\begin{equation}
\label{prelem}
F(q) = \left( q^r,q^{m-r};q^m\right)_{\infty}^{-1} = -iq^{\frac{m}{12} -\frac{r}{2}} \frac{\eta(m \tau)}{\Theta(r\tau; m\tau)}  .
\end{equation}
\noindent Letting $w \to r\tau$ and $\tau \to m \tau$ in \eqref{thetatrans}, 
\begin{equation}
\Theta(r\tau; m\tau) = \frac{ i }{\sqrt{-im \tau}} e^{-\frac{\pi i r^2\tau}{m}} \Theta\left( \frac{r}{m};-\frac{1}{m \tau}\right),
\end{equation}
where by plugging in $\frac{r}{m}$ for $w$ and $\frac{-1}{m \tau}$ for $\tau$ in \eqref{jacobi}, we have 
\begin{equation}
\Theta\left( \frac{r}{m};\frac{-1}{m \tau}\right)  = -i e^{-\frac{\pi i}{4m \tau} - \frac{\pi i r}{m}} \left( e^{-\frac{2 \pi i }{m \tau}};e^{-\frac{2 \pi i }{m \tau}}\right)_{\infty} \left( e^{ \frac{2\pi ir }{m}}; e^{ -\frac{2 \pi i }{m \tau}}\right)_{\infty} \left( e^{ - \frac{ 2\pi i }{m}\left( r + \frac{1}{\tau}\right)};e^{ -\frac{2 \pi i }{m \tau}}\right)_{\infty}.
\label{thets}
\end{equation}

\noindent Using the inversion formula \eqref{etatrans}, we also have that 
\begin{equation}
\eta(m \tau) = \frac{1}{\sqrt{-im\tau}} e^{\frac{ - \pi i}{12m \tau}} \left( e^{\frac{-2\pi i }{m \tau}};e^{\frac{-2\pi i }{m \tau}}\right)_{\infty}.
\label{newded}
\end{equation}
Finally, plugging \eqref{thets} and \eqref{newded} into \eqref{prelem}, we have 

\begin{equation}
F(q)= \frac{e^{\pi i \tau \left( \frac{m}{6}-r + \frac{r^2}{m}\right) + \frac{\pi i }{6m \tau}}} {i e^{\frac{-\pi i r}{m}}\left( 1- e^{\frac{2\pi i r}{m}}\right)\left( e^{\frac{2\pi i r}{m} - \frac{2\pi i r}{m \tau}}, e^{\frac{-2\pi i r}{m}- \frac{2\pi i }{m \tau}}; e^{\frac{-2\pi i}{m \tau}} \right)_{\infty}}.
\end{equation}
\end{proof}
\indent 
False theta functions do not satisfy the same functional equations as the typical theta function, although they have the same shape. Euler-Maclaurin summation can be used to expand $L(q)$ out in desired powers of $q$, but we use a result found in
\cite{LSource}, where the authors explored the function
\begin{equation}
f_{a,b}(\tau) = \sum_{ n \geq 1} (-1)^n q^{(an^2 +bn)/2}
\label{fab}
\end{equation} 
and showed the following:
\begin{lemma}[Kim, Kim, Seo (2015)]\label{BB}\spa
For $\tau = x+i y$, as $y \to 0^+$ with $|x|\leq y \leq \frac{\sqrt{3}}{8}$,
\begin{equation}
\left| f_{a,b}(\tau) - \left( -\frac{1}{2} + \frac{b}{8}(-2 \pi i \tau) + \frac{ab}{32}(-2\pi i \tau)^2 + \frac{b(6a^2-b^2)}{384}(-2\pi i \tau)^3 \right)\right| < cy^4.
\end{equation}
where $c=105\pi^4a^4b^{14} e^{\frac{\pi \sqrt{3}b^2}{32a}}$.
\end{lemma}
If we let $a \mapsto m, b \mapsto (-4r-m)$ in \eqref{fab}, then 
\begin{equation}
 L(q) = \sum_{n \geq 0} (-1)^{n} q^{-2rn + \frac{mn}{2}(n+1)} = -q^{2r}f_{m,-m-4r}(q). 
\end{equation}
Now letting $z = -2\pi i \tau$ and $x=e^{-z}$, we can find constants for  the false theta function in powers of $z$ up to a desired degree $k$, as in Lemma 1 of \cite{Wright2}.
\begin{lemma} For $|\text{arg}(z)|< \frac{\pi}{2} - c$ where $c$ is some arbitrary constant, taking  $|z| \to 0$ gives 
\begin{equation}
L(x) = \alpha_0 -\sum_{s=1}^{k-1} \alpha_sz^s + O(z^k),
\end{equation}
\label{LEM}

where 
\[
\alpha_s \sim  \frac{(4r+m)(m^{s-1})}{2^{2s+1}}\hbox{for } 1 \leq s \leq k-1, \hbox{ and }  \alpha_0 := \frac12.
\]
\end{lemma}
\section{Wright's Circle Method and Proof of Theorem \ref{Theorem1}}\label{lastSection} The proof follows the notation and method of Wright in \cite{Wright2}.  We let $x=e^{-z}$ replace $q=e^{2\pi i \tau}$ and send $\tau \mapsto \frac{iz}{2\pi }$. Then the modular transformation in Proposition \ref{PROP1} becomes
\begin{equation}
F(x) = \frac{1}{2\sin\left( \frac{\pi r}{m}\right)} e^{-\left( \frac{m}{12}+\frac{r^2}{2m} - \frac{r}{2}\right)z + \frac{\pi^2}{3mz} }\left( 1+ O\left( e^{\frac{-4\pi^2}{mz}}\right) \right).\label{NEWINV}
\end{equation}
Defining 
\[
w(z) :=  \frac{1}{2\sin\left( \frac{\pi r}{m} \right) }e^{\left( \frac{r(m-r)}{2m} -\frac{m}{12}\right) z + \frac{ \pi^2 }{3mz}},
\]
the following term is critical for the Circle Method
\[
w(z)x^{-n} = \frac{1}{2\sin\left( \frac{\pi r}{m} \right) } e^{\left( \frac{r(m-r)}{2m} -\frac{m}{12} +n\right) z + \frac{ \pi^2 }{3mz}}.
\]
Using the Arithmetic-Geometric Mean Inequality on the exponent above gives
\[
\left( \frac{r(m-r)}{2m} -\frac{m}{12} +n\right) z + \frac{ \pi^2 }{3mz} \geq 2 \sqrt{\frac{\pi^2}{3m} \left(\frac{r(m-r)}{2m} -\frac{m}{12} +n\right)},
\]
where equality occurs when 
$z= \kappa$ with
\[
\kappa := \frac{ \pi }{\sqrt{\frac{3r(m-r)}{2} - \frac{m^2}{4} +3mn}}.
\]
\\
\indent 
The circle $\cc$ of radius $e^{-\kappa}$ centered at $0$ will be the contour of integration. The major arc $\cc_1$ is defined to be where $|$arg$(x)| \leq \rho \kappa$ for small fixed $\rho >0$. The minor arc is the rest of the circle $\cc_2 = \cc-\cc_1$, away from the pole at $x=1$. 
\\
\indent 
We can now use the Circle Method over $\cc$, establishing the integral
\begin{equation}
h(n) = \frac{1}{2\pi i } \int_{\cc} \frac{L(x)x^{-(n+1)}}{\left(x^r;x^m\right)_{\infty}\left(x^{m-r};x^m\right)_{\infty}} dx,
\end{equation} which will be approximated by $\sum_{s\geq 1} \alpha_s h_s(n)$ with the constants $\alpha_s$ from Lemma \ref{LEM} and 
\begin{equation}
h_s(n) := \frac{1}{2\pi i } \int_{\cc_1} \frac{z^sw(z)}{x^{n+1}}dx.
\end{equation}

The main term of each $h_s(n)$  will be shown to be $e^{2N}$ where
\[
N := \frac{\pi^2}{3m \kappa} = \pi\sqrt{\frac{r(m-r)}{6m^2}-\frac{1}{36}+\frac{n}{3m}}.
\]
\begin{proposition}
\label{PROP2} For $m \geq 3$, $0 \leq r< m$, and as $n \to \infty$, 
\[
h(n) = \sum_{ s=0}^{k-1} \alpha_s h_s(n) + O\left(N^{-k}e^{2N}\right).
\]
\end{proposition}
\begin{proof}
Breaking up the difference of the integrals into 3 pieces, we have
\begin{align*}
 2\pi i\left( h - \sum_{s=0}^{k-1} \alpha_s h_s\right) = \int_{\cc_1} \left( L(x) - \sum_{s=0}^{k-1} \alpha_s z^s \right) w(z) x^{-n-1}dx \spa \spa &(E_1)\\
+\int_{\cc_1} L(x) \left((x^r, x^{m-r};x^m)_{\infty}^{-1} - w(z)\right)x^{-n-1} dx \spa \spa &(E_2)\\
+ \int_{\cc_2} L(x)(x^r,x^{m-r};x^m)_{\infty}^{-1}x^{-n-1}dx.\spa \spa &(E_3) 
\end{align*} 
\indent Note that 
\begin{align*}
\left|w(z) x^{-n} \right|  
&\leq \left| e^{\left( \frac{r(m-r)}{2m} -\frac{m}{12} +n\right) z} \right| \left| e^{\frac{ \pi^2 }{3mz}}\right|, \nonumber 
\end{align*}
and letting $z = \kappa + i \nu$ on the circle $\cc$, 
the real part of $\frac1z$ is $\frac{\kappa}{\kappa^2 + \nu^2}$, so 
\begin{align}
\frac{1}{2\sin \left( \frac{\pi r}{m} \right)} \left| e^{\left( \frac{r(m-r)}{2m} -\frac{m}{12} +n\right) z} \right| \left| e^{\frac{ \pi^2 }{3mz}}\right| &\ll e^{\left( \frac{r(m-r)}{2m} -\frac{m}{12} +n\right)\kappa+ \frac{\pi^2}{3m}\left( \frac{\kappa}{\kappa^2 + \nu^2}\right)}\\
&\ll  e^{2\pi\sqrt{\frac{r(m-r)}{6m^2}-\frac{1}{36}+\frac{n}{3m}}} \nonumber
\\&=O\left(e^{2N}\right)
\label{this}.
\end{align}
Over $\cc_1$, we see that $|z| \leq | \kappa + \rho \kappa i | = O( \kappa)$. Along with Lemma \ref{LEM}, 
\begin{align}
\label{that}
\left| L(x) -\sum_{s=0}^{k-1} \alpha_s z^s \right| \ll |z|^k \ll \kappa^k.
\end{align} 
Using \eqref{this} and \eqref{that}, 
\begin{align}
|E_1| &= \left|\int_{\cc_1} \left( L(x) - \sum_{s=0}^{k-1} \alpha_s z^s \right) w(z) x^{-n-1}dx\right| \nonumber\\
&\ll \left|e^{z}\right|\left|w(z)x^{-n}\right|\left| L(x) - \sum_{s = 0}^{k-1} \alpha_s z^s\right| \nonumber \\
&= O(\kappa^{k} e^{2N})= O(N^{-k}e^{2N}).
\end{align}
\indent 
The modular inversion formula in \eqref{NEWINV} shows that  
\begin{align}
\label{these}
\left| \bigg( \left(x^r,x^{m-r};x^m\right)_{\infty} - w(z)\bigg)\right| \ll \left| w(z) e^{\frac{-2 \pi^2}{mz}}\right|.
\end{align}
Asymptotically, \eqref{that} really tells us that on $\cc_1$
\begin{align}
\label{those}
|L(x)| \leq \left| \sum_{s=0}^{k-1} \alpha_s z^s  + cz^k \right|  = O(1).
\end{align}
Hence we can bound $E_2$ by using \eqref{this}, \eqref{these},  and \eqref{those}:
\begin{align}
|E_2| &= \left| \int_{\cc_1} L(x) \left((x^r, x^{m-r};x^m)_{\infty}^{-1} - w(z)\right)x^{-n-1} dx\right|\nonumber \\
&\ll |L(x)|\left| w(x)x^{-n}\right| \left| e^{\frac{-2\pi^2}{mz}}\right| \nonumber \\
&\leq e^{2N - \frac{4\pi^2}{\kappa\sqrt{1+\rho^2}}}\nonumber\\
 &=O(N^{-k} e^{2N}).
 \end{align}
\indent As we move away from the pole at $x=1$ along the minor arc, there are still some problematic spots, specifically at $m$th roots of unity. Firstly, suppose $x$ on $\cc_2$ is not close to a root of unity. Explicitly, we have that $x = e^{-\kappa +i \nu}$ and want to insure that $\nu$ is sufficiently far away from a root of unity, thus we require for $1\leq \ell \leq m-1$ that 
\begin{equation}
\label{nuConditions}
\left| \nu - \frac{2\pi i \ell}{m} \right| > \rho \kappa.
\end{equation}
Now we can use the logarithm for an asymptotic expansion:
\begin{align*}
\log(F(x)) &= \log \left( \left(x^r,x^{m-r};x^m\right)_{\infty}^{-1}\right)\\
&= - \sum_{ n \geq 1} \log \left(1- x^{mn-r}\right)+\log \left( 1-x^{(n-1)m +r}\right).
\end{align*}
The Taylor Series Expansion for $\log(1-x^k)$ is 
\begin{align*}
-\log(1-x^k) = \sum_{ n \geq 1} \frac{x^{nk}}{n},
\end{align*}
so the above becomes 
\begin{align*}
 - \sum_{ n \geq 1} \log \left(1- x^{mn-r}\right)+\log \left( 1-x^{(n-1)m +r}\right) &= 
 \sum_{ n \geq 1} \sum_{j \geq 1} \frac{ x^{((n-1)m + r)j} + x^{(mn-r)j}}{j} \\
 &= \sum_{ j \geq 1} \frac{x^{rj}+x^{(m-r)j}}{j} \sum_{ n \geq 0} \left( x^{mj}\right)^n.
 \end{align*}
 The sum on the right is a geometric series, so 
 \begin{align}
 \log \left( F(x)\right) = \sum_{ j \geq 1} \frac{x^{rj}+x^{(m-r)j}}{j\left( 1- x^{mj}\right) }.
\end{align}
Taking the absolute value, we can bound the log term: 
\begin{align}
|&\log(F(x))|= \left|\sum_{ j \geq 1} \frac{x^{rj}+x^{(m-r)j}}{j\left( 1- x^{mj}\right) }\right| 
\leq \sum_{j \geq 1} \frac{\left|x^{rj} \right|+ \left|x^{(m-r)j} \right|}{j \left|1-x^{mj} \right|} \nonumber \\
&\leq \sum_ { j \geq 2} \frac{ \left| x^{rj}\right| + \left| x^{(m-r)j} \right|}{j \left| 1-x^{mj}\right|} + \frac{ |x^r| + |x^{m-r}|}{|1-x^m|} + \left( |x^r| + |x|^{m-r}\right) \left( \frac{1}{1-|x|^m} - \frac{1}{1-|x|^m} \right) \nonumber ,\nonumber
\end{align}
\noindent where we have pulled out the $j=1$ term and added and subtracted in  a $\frac{|x^r|+|x^{m-r}|}{1-|x|^m}$ term. As $|1-x^{mj}|\geq (1-|x^{mj}|)$, we rearrange the sum and see
\begin{align}
& \sum_ { j \geq 2} \frac{ \left| x^{rj}\right| + \left| x^{(m-r)j} \right|}{j \left| 1-x^{mj}\right|} + \frac{ |x^r| + |x^{m-r}|}{|1-x^m|} + \left( |x^r| + |x|^{m-r}\right) \left( \frac{1}{1-|x|^m} - \frac{1}{1-|x|^m} \right) \nonumber \\
&\leq \sum_{ j \geq 1} \frac{ |x^{rj}| + |x^{(m-r)j}|}{j(1-|x^{mj}|)} + \left( |x^r| + |x^{m-r}|\right) \left( \frac{1}{|1-x^m|}-\frac{1}{1-|x|^m}\right)\nonumber \\
& = \log\left( (|x^r|, |x^{m-r}|; |x^m|)_{\infty}^{-1} \right) + \left( |x^r| + |x^{m-r}|\right) \left( \frac{1}{1-|x|^m} - \frac{1}{|1-x^m|} \right).\label{log1}
\end{align}
Now $|x| =e^{-\kappa}$, so by our modular inversion formula \eqref{NEWINV},
\begin{equation} 
\label{log2}
\log \left((|x^r|,|x^{m-r}|;|x^m|)_{\infty}^{-1}\right) \sim \log(w(\kappa)) \sim  \frac{ \pi^2 }{3m\kappa}.
\end{equation}
On $\cc_2$, we know $|\text{arg} (x)| > \rho \kappa$, so
\begin{equation}
\label{log3}
\frac{1}{|1-x^m|}\ll \frac{1}{m\kappa \sqrt{1+\rho^2}},
\end{equation}
and by using the Taylor expansion of $1-|x^m|\leq 1-e^{-m\kappa} \sim m\kappa$, one can approximate
\begin{equation}
\label{log4}
\frac{1}{1-|x|^m} \sim \frac{1}{m \kappa}.
\end{equation}
Plugging \eqref{log2}, \eqref{log3}, and \eqref{log4} into \eqref{log1}, we have that 
\begin{equation}
|\log(F(x))| \sim \frac{\pi^2}{3m\kappa} - \frac{1}{m\kappa}+ \frac{1}{m \kappa \sqrt{1+\rho^2}}.
\end{equation}
Finally, we can bound $E_3$: 
\begin{align}
|E_3| &= \left| \int_{\cc_2}L(x)F(x)x^{-n-1} dx\right| \nonumber \\
&\ll |L(x)x^{-n}||\text{exp}(\log(F(x)))|\nonumber \\
&\ll e^N \text{exp}\left(\frac{1}{m\kappa}\left( \frac{\pi^2}{3}-1+\frac{1}{\sqrt{1+\rho^2}}\right)\right)\nonumber\\
&=\text{exp}\left(N+\left(N-\frac{3N}{\pi^2} +\frac{3N}{\pi^2\sqrt{1+\rho^2}} \right)\right)\nonumber \\
&=O(e^{2N-c})  \spa \spa  \hbox{for some } c>0.
\end{align}
The goal was to beat the bound $e^{2N}$, hence we have motivated the choice for the constant $\rho >0$ and have successfully bounded $E_3$ under the assumption that $x$ was away from a root of unity.
\\ \indent
Suppose now that $x$ is close to an $m$th root of unity and $\text{gcd}(r,m) =1$ (otherwise we can reduce the case down).  Following the techniques used by Bringmann and Mahlburg in \cite{Schur2}, we want to shift away from the root of unity by a small amount that still allows us to use the periodicity of the theta functions.
\\
\indent The $m$th roots of unity for $x= e^{-z}$ are near $z= \kappa - \frac{2\pi i \ell}{m}$ for $1 \leq \ell \leq m-1$. Thus we consider $z=\kappa -\frac{2\pi i a}{m}$ for some $a$ so that $\left| \frac{2\pi i}{m}(a-\ell)\right| < \rho \kappa$. 
With this shift in mind, we can rewrite
\begin{align}
x = e^{-z} &= e^{-\kappa + \frac{2\pi i \ell}{m} + \frac{ 2\pi i (a-\ell)}{m}} \nonumber \\
&= \zeta_m^k e^{-\kappa + \frac{2\pi i (a-\ell)}{m}}\nonumber 
\\
&= \zeta_m^k e^{-z'} \label{Ploog}
\end{align}
where $\zeta_m = e^{\frac{2\pi i }{m}}$ and $z' = \kappa - \frac{2\pi i (a-\ell)}{m}$. Then we can write $F(x) $ in terms of $z'$:
\begin{align}
F(x) &= (x^r, x^{m-r};x^m)_{\infty}^{-1}\\
&=\left( \zeta_m ^{(\ell r \pmod{m})}e^{-rz'}, \zeta_m^{(-\ell r \pmod{m})}e^{-(m-r)z'}; e^{-mz'}\right)_{\infty}^{-1} \nonumber \\
& = \left( e^{ 2\pi i \phi + 2 \pi i \frac{i r z'}{2 \pi}}, e^{-2\pi i \phi + 2\pi i \frac{i(m-r)z'}{2\pi }}; e^{ 2\pi i \frac{ imz'}{2\pi}} \right)_{\infty}^{-1}
\end{align}
where $ \phi := \left\{ \frac{\ell r}{m} \right\} = \frac{\ell r}{m} - \lfloor \frac{\ell r}{m} \rfloor$
 is the fractional part (this will be nonzero since gcd$(r,m) = 1$). Then we can express this in terms of theta and eta functions:
\begin{align}
F(x) &= \frac{-i e^{ \frac{ \pi i}{4} - \frac{imz'}{2\pi}} e^{- \pi i ( \phi + \frac{irz'}{2\pi })} \left( e^{\frac{ i m z'}{2 \pi }}; e^{\frac{imz'}{2\pi}}\right)_{\infty}}{\Theta\left( \phi + \frac{irz'}{2 \pi}, \frac{imz'}{2\pi}\right)}\nonumber \\
&= \frac{-i e^{\frac{\pi i }{6} \frac{imz'}{2\pi} - \pi i ( \phi + \frac{irz'}{2\pi})} \eta\left( \frac{imz'}{2\pi}\right)}{ \Theta \left( \phi + \frac{irz'}{2\pi}, \frac{imz'}{2\pi}\right)}.\label{shifty}
\end{align}
Now using the transformations of the theta \eqref{thetatrans} and eta \eqref{etatrans} functions in \eqref{shifty}, we have
\begin{align}
F(x) &= \frac{-e^{\pi i ( \frac{iz'}{\pi}( \frac{m}{12}-\frac{r}{2}) - \phi) + \frac{2\pi ^2}{mz'} ( \phi^2 + \phi i r z' - \frac{r^2z'^2}{4\pi^2})} \eta \left( \frac{-2\pi }{imz'}\right)}{\Theta \left( \frac{-2\pi i \phi}{m z'} + \frac{r}{m}; \frac{2\pi i }{mz'}\right)}.\label{Rewrite}
\end{align}
For the eta function, the approximation is
\begin{equation}
\label{eter}
\eta\left( \frac{-2\pi}{imz'}\right)  \sim e^{\frac{\pi i }{12} \cdot \frac{-2\pi }{imz'}} = e^{\frac{-\pi ^2}{6mz'}}
\end{equation}
and  for the theta function,
\begin{align}
\Theta \left( \frac{-2\pi i \phi}{mz'} + \frac{r}{m}; \frac{2\pi i }{mz'}\right) &\sim -i e^{ \frac{\pi i }{4} \frac{ 2\pi i }{mz'} - \pi i ( \frac{-2\pi i \phi}{mz'} + \frac{r}{m})} ( 1- e^{2\pi i ( \frac{ -2\pi i \phi}{mz'} +\frac{r}{m})}) \nonumber\\
&= -ie ^{\frac{-\pi ^2}{2m z'} - \frac{2\pi^2 \phi}{mz'}}\zeta_{2m}^{-r} \left(1 - \zeta_{m}^r e^{\frac{4\pi ^2 \phi}{mz'}}\right)\label{theter}.
\end{align}
Putting \eqref{eter} and \eqref{theter} into \eqref{Rewrite}, the main exponential term becomes
\[
e^{\frac{2\pi^2 \phi^2}{mz'} - \frac{\pi^2}{6mz'} + \frac{\pi^2}{2mz'} - \frac{2\pi^2 \phi}{mz'}} = e^{ \frac{\pi^2}{mz'} \left( \frac13 - 2\phi ( 1-\phi)\right)}.
\]
Bounding this term and noticing that $2\phi(1-\phi) >0$ for $\text{gcd}(r,m)=1$,
\begin{align*}
\left| e^{ \frac{\pi^2}{mz'} \left( \frac13 - 2\phi ( 1-\phi)\right)}\right| &\leq e^{\text{Re}\left( \frac{1}{z'}\right) \frac{\pi^2}{m} \left( \frac13 - 2\phi(1-\phi)\right)} \\
&\leq e^{\frac{\pi^2}{\kappa m} (\frac13 - 2\phi(1-\phi))} \\
&= e^{N(2-c)} \spa \spa  \hbox{for some } c>0,
\end{align*}
hence we still receive our exponential savings over $e^{2N}$. 
\end{proof}
\indent Lastly, we need to bound the approximating integral $h_s$, and calculate out the final asymptotic result. 
\begin{proposition} 
\label{PROP3}
\added{For $m \geq 3$ and $0 \leq r< m$, }
\[ 
h_s(n)= \frac{\csc{\frac{\pi r} {m}}}{4} \kappa^{s+1} I_{-s-1}(2N) + O\left(e^{\frac{(\rho^2+2)N}{\rho^2+1}}\right). 
\]
\end{proposition}
\begin{proof}
Let $\dd$ be the rectangle whose endpoints are $\pm \kappa \pm \rho\kappa i$ (traversed counter-clockwise). Then we define
\begin{align*}
W_s &:= \frac{\csc{\frac{\pi r} {m}}}{4\pi i } \int_{\dd} z^s w(z) e^{nz} dz\\
&= \int_{- \kappa-\rho\kappa i}^{\kappa-\rho\kappa i} + \int_{\kappa-\rho\kappa i }^{\kappa+\rho\kappa i} + \int_{\kappa+\rho\kappa i }^{ -\kappa + \rho\kappa i} + \int_{-\kappa +\rho\kappa i}^{-\kappa -\rho\kappa i}\dots dz\\
&= \int_{\dd_1} + \int_{\dd_2} + \int_{\dd_3} +\int_{\dd_4},
\end{align*}
since the integral over $\dd_2$ is exactly the image of $\cc_1$ in the $z-$plane, hence
\[
 h_s(n) =\frac{1}{2\pi i } \int_{\cc_1} \frac{z^sw(z)}{x^{n+1}}dx= \frac{\csc{\frac{\pi r} {m}}}{4\pi i } \int_{\dd_2} z^s w(z) e^{nz} dz.
\]

\noindent Sending $z \mapsto \kappa t$, \added{the integral representation for the modified Bessel function $I_{\alpha}(x)$ appears (obtained by using the transformation $I_{\alpha}(x) = i^{-a} J_{\alpha}(ix)$ on the loop integral for $J_{\alpha}(x)$ found on page 355 of  \cite{Modern}):}
\begin{align*}
W_s
&=\frac{\csc{\frac{\pi r} {m}}}{8\pi i } \kappa^{s+1} \int_{\dd'} t^s e^{N(t+ \frac1t)}dt \\
&= \frac{\csc{\frac{\pi r}{m}}}{4}\kappa^{s+1} I_{-s-1}(2N),
\end{align*}
In the $t-$plane on the contour $\dd'$, the endpoints of the rectangle are now given by $\pm 1\pm \rho i$. Consider the integral on $\dd_1'$:
\begin{align*}
 \left| \int_{-1-\rho i}^{1-\rho i} t^s e^{N(t+\frac1t)}dt \right|\leq \int_{-1-\rho i}^{1-\rho i} |t^s| \left| e^{N(t+\frac1t)} \right|dt.
\end{align*}
If we let $t = u -\rho i$, then $ -1 \leq u \leq 1$, so $Re\left( t + \frac1t\right) = u + \frac{u}{\rho^2+u^2}$. This function is increasing on the interval, so it achieves a maximum value when $u=1$. Thus
\[
 \left| \int_{-1-\rho i}^{1-\rho i} t^s e^{N(t+\frac1t)}dt \right| \ll e^{\frac{(\rho^2+2)N}{\rho^2+1}}.
\]
Similarly on $\dd_3$ and $\dd_4$, the integrals are $O\left(e^{\frac{(\rho^2+2)N}{\rho^2+1}}\right)$. 
\end{proof}
\begin{remark*}
As $\frac{\rho^2+2}{\rho^2 +1} < 2$ for all $\rho \neq 0$, the error term is indeed smaller than the main term.
\end{remark*}
We have $\kappa = \frac{3mN}{\pi^2}$. If we use Hankel's approximation for the Bessel function (for a good reference, see \cite{Abram}), we can write
\begin{align}
&h_0(n) = \frac{\csc{\frac{\pi r} {m}}}{4} \kappa I_{-1}(2N) + O\left(e^{\frac{(\rho^2+2)N}{\rho^2+1}}\right) \nonumber \\
&= \frac{\csc{\frac{\pi r}{m}} \pi^{3/2}e^{2N}}{24 m N^{3/2}}\left[ 1- \frac{ 4s^2 + 8s+3}{2^4N} + \frac{(4s^2+8s+3)(4s^2 +8s -5)}{2!(2^4N)^2}\dots \right] \nonumber
\\
&\hspace{10cm}+ O\left(e^{\frac{(\rho^2+2)N}{\rho^2+1}}\right).\label{Bessie}
\end{align}
From Proposition \ref{PROP2}, 
\begin{equation}
h(n) = \sum_{s=0}^k \alpha_s h_s(n) + O \left( N^{-k-1}e^{2N}\right)\label{ARG},
\end{equation}
and plugging \eqref{Bessie} into \eqref{ARG} and pulling out the first term, 
\begin{equation}
h(n)=\frac{\csc{\frac{\pi r}{m}}e^{2N} \pi^{3/2}}{24 m N^{3/2}}\left(\sum_{s\geq 1}\sum_{r \geq 0}\alpha_{s}\beta_{r,s}  + O(N^{-k})\right).
\end{equation}
for constants $\beta_{r,s}$. Finally, $N = \pi\sqrt{\frac{r(m-r)}{6m^2}-\frac{1}{36}+\frac{n}{3m}}$, and when $s=0$, the coefficients $\alpha_s\beta_{r,s} =1$, hence
\begin{equation}
h(n) = \frac{\csc{\frac{\pi r}{m}}}{2^3 3 m \left(\frac{r(m-r)}{6m^2}-\frac{1}{36}+\frac{n}{3m}\right)^{3/4}}e^{\left(2\pi\sqrt{\frac{r(m-r)}{6m^2}-\frac{1}{36}+\frac{n}{3m}}\right)}\left( 1+ O(N^{-1})\right).
\end{equation}
As $s_{(r,m)}(n) \sim h(n)$, we have effectively determined the asymptotic expansion for mixed congruence stacks. The terms $\frac{r(m-r)}{6m^2} - \frac{1}{36}$ are insignificant compared to $\frac{n}{3m}$, so we can drop them, completing the proof of Theorem \ref{Theorem1}.


\begin{thebibliography}{99}
\bibitem{Abram} M. Abramowitz and I.A. Stegun,``Modified Spherical Bessel Functions." §10.2 in \emph{Handbook of Mathematical Functions}, $9$th printing. New York: Dover, pp. 443-445, 1972. 
\bibitem{Andrews} G. Andrews, ``An Introduction to Ramanujan's `Lost' Notebook", \emph{Ramanujan: Essays and Surveys (AMS)} {\bf 22}, Providence, RI (2001) 165-184

\bibitem{Andrews2} G. Andrews, \emph{The Theory of Partitions}, Cambridge University Press, Cambridge, 1998.

\bibitem{Andrews3} G. Andrews, ``Concave and Convex Compositions," \emph{Ramanujan J.} {\bf 31} (2013), no. 1-2, 67-82.

\bibitem{Andrews4} G. Andrews, ``Ramanujan's `lost' notebook I: Partial $\theta$-functions." \emph{Adv. in Math.} 41 (1981), no. 2, 137-172. 

\bibitem{Auluck} F. Auluck, ``On Some New Types of Partitions Associated with Generalized Ferrers Graphs," \emph{Math. Proc. Cambridge Philos. Soc.}, {\bf 47} (1951) 679-686.



\bibitem{BM} O. Beckwith and M. H. Mertens , ``The Number of Parts in Certain Residue Classes of Integer Partitions", Res. Number Theory {\bf 1} no. 11 (2015) 1-15.



\bibitem{Stacks} K. Bringmann and K. Mahlburg,  ``Asymptotic Formulas for Stacks and Unimodal Sequences," \emph{J. Combin. Theory Ser. A} {\bf 126} (2014) 194-215.

\bibitem{Schur2} K. Bringmann and K. Mahlburg, ``Schur's Second Partition Theorem and Mixed Mock Modular Forms", arXiv:1307.1800 [math.NT], 2013.

\bibitem{Dave} H. Davenport,  \emph{Multiplicative Number Theory}, $2$nd ed. New York: Springer-Verlage. 1980.

\bibitem{Hyper} G. Gasper and M. Rahman, \emph{Basic Hypergeometric Series}, Second Edition. Encycl. of Math. and its Appl., {\bf 96}. Cambridge University Press, Cambridge, 2004.  
\bibitem{HR}G.H. Hardy and S. Ramanujan, ``Asymptotic Formulae in Combinatory Analysis," \emph{Proc. Lond. Math. Soc.} (2) {\bf 17}, 75-115 (1919).
\bibitem{IW} H. Iwaniec and E. Kowalski,  \emph{Analytic Number Theory}, Colloquium Publications (AMS); Volume 53, pp. 449-467. 
\bibitem{LSource} B. Kim, E. Kim, and J. Seo, ``Asymptotics for q-expansions involving partial theta functions." \emph{Discrete Math.} 338 (2015), no.2, 180-189.
\bibitem{Living} J. Livingood, ``A partition function with the prime modulus $p >3$," \emph{American Journal of Mathematics} {\bf 67}, 194-208. (1945) 
\bibitem{Mein} G. Meinardus, ``Asymptotische Aussagen {\"u}ber Partitionen," \emph{Mathematische Zeitschrift}, {\bf 59}, 388-398. (1954)

\bibitem{Road}  T.H. Ngo and R.C. Rhoades, ``Integer partitions, probabilities, and quantum modular forms", preprint, available at \url{http://math.stanford.edu/~rhoades/RESEARCH/papers.html}.
\bibitem{Rog} L.J. Rogers, ``On two theorems of combinatory analysis and some allied identities", \emph{Proc. London Math. Soc.} (2), 16 (1916) 315-336.

\bibitem{Temp} H.V. Temperley, ``Statistical Mechanics and the partition of numbers II. The form of crystal surfaces", \emph{Proc. Camb. Phil. Soc.} {\bf 48} (1952), 683-97.

\bibitem{Modern} \added{E.T. Whittaker and G.N. Watson. \emph{A Course in Modern Analysis.} Cambridge, 1948, p. 355.}

\bibitem{Wright2} E.M. Wright, ``Stacks (II)", \emph{Quart. J. Math.} (Oxford) {\bf 22} (1971) 107-16.

\bibitem{WrightAS} E.M. Wright, ``Asymptotic Partition Formulae II: Weighted Partitions," \emph{Proc. London Math. Soc.} (92) {\bf 36} (1933), 117-141.






















\end{thebibliography}
\end{document}